\documentclass[a4paper,12pt]{article}
\usepackage{latexsym}
\usepackage{amsmath}
\usepackage{amssymb}
\usepackage{amscd}
\usepackage{array}
\usepackage{comment}
\usepackage[all]{xy}
\usepackage{amsthm}
\usepackage{amsfonts}
\usepackage{enumerate}
\newtheorem{theorem}{Theorem}[]
\newtheorem{proposition}[theorem]{Proposition}

\newtheorem{corollary}[theorem]{Corollary}

\theoremstyle{definition}
\newtheorem{definition}[theorem]{Definition}

\newtheorem{remark}[theorem]{Remark}

\newcommand{\Q}{\mathbb Q}

\newcommand{\HH}{\mathcal H}

\newcommand{\Gal}{\mathrm{Gal}}

\newcommand{\Hol}{\mathrm{Hol}}

\newcommand{\Sym}{\operatorname{Sym}}

\newcommand{\End}{\operatorname{End}}

\newcommand{\Aut}{\operatorname{Aut}}

\newcommand{\wK} {{\widetilde{K}}}

\title{Induced Hopf Galois structures}
\author{Teresa Crespo\footnote{Departament d'\`Algebra i Geometria, Universitat de Barcelona, Gran Via de les
Corts Catalanes 585, E-08007 Barcelona, Spain, e-mail: teresa.crespo@ub.edu}, Anna Rio\footnote{Departament de Matem\`{a}tica Aplicada II, Universitat Polit\`{e}cnica de Catalunya, C/Jordi Girona, 1-3- Edifici Omega, E-08034 Barcelona, Spain, e-mail: ana.rio@upc.edu} and Montserrat Vela\footnote{Departament de Matem\`{a}tica Aplicada II, Universitat Polit\`{e}cnica de Catalunya, C/Jordi Girona, 1-3- Edifici Omega, E-08034 Barcelona, Spain, e-mail: montse.vela@upc.edu}}
\date{\today}

\begin{document}

\maketitle

\begin{abstract} For a finite Galois extension $K/k$ and an intermediate field $F$ such that $\Gal(K/F)$ has a normal complement in $\Gal(K/k)$, we construct and characterize Hopf Galois structures on $K/k$ which are induced by a pair of Hopf Galois structures on $K/F$ and $F/k$.
\end{abstract}

\section{Introduction}

A finite extension of fields $K/k$ is a Hopf Galois extension if
there exist a finite cocommutative $k$-Hopf algebra $\HH$  and a
Hopf action of $\HH$ on $K$, i.e a $k$-linear map $\mu: \HH\to
\End_k(K)$ inducing a bijection $K\otimes_k \HH\to\End_k(K)$. We
shall call such a pair $(\HH,\mu)$ a Hopf Galois structure on $K/k$.
Hopf Galois extensions were introduced by Chase and Sweedler in \cite{C-S}.
For separable field extensions, Greither and
Pareigis \cite{G-P} give the following group-theoretic
characterization of the Hopf Galois property.

\begin{theorem}
Let $K/k$ be a separable field extension of degree $n$, $\wK$ its Galois closure, $G=\Gal(\wK/k), G'=\Gal(\wK/K)$. Then
$K/k$ is a Hopf Galois extension if and only if there exists a
regular subgroup $N$ of $S_n$ normalized by $\lambda (G)$, where
$\lambda:G \rightarrow S_n$ is the morphism given by the action of
$G$ on the left cosets $G/G'$.
\end{theorem}

For a given Hopf Galois structure on $K/k$, we will refer to the isomorphism class of the corresponding group $N$ as the type of the Hopf Galois
structure. The Hopf algebra $\HH$ corresponding to a regular subgroup $N$ of $S_n$ normalized by $\lambda (G)$ is the subalgebra of the group algebra $\wK[N]$ fixed under the action of $G$, where $G$ acts on $\wK$ by $k$-automorphisms and on $N$ by conjugation through $\lambda$. It is known that the Hopf subalgebras of $\HH$ are in 1 to 1 correspondence with the subgroups of $N$ stable under the action of $G$. In the sequel for an extension $K/k$ endowed with a Hopf Galois structure with Hopf algebra $\HH$ corresponding to a group $N$, and $N'$ a $G$-stable subgroup of $N$, we shall denote by $K^{N'}$ the subfield of $K$ fixed by the Hopf subalgebra of $\HH$ corresponding to $N'$ and refer to it as the subfield of $K$ fixed by $N'$ (see \cite{CRV} Theorem 2.3).

Childs \cite{Ch1} gives an equivalent more effective condition to the Hopf Galois property introducing the holomorph of the regular subgroup $N$ of $S_n$. We state the more precise formulation of this result due to Byott \cite{B} (see also \cite{Ch2} Theorem 7.3).

\begin{theorem}\label{theoB} Let $G$ be a finite group, $G'\subset G$ a subgroup and $\lambda:G\to \Sym(G/G')$ the morphism given by the action of
$G$ on the left cosets $G/G'$.
Let $N$ be a group of
order $[G:G']$ with identity element $e_N$. Then there is a
bijection between
$$
{\cal N}=\{\alpha:N\hookrightarrow \Sym(G/G') \mbox{ such that
}\alpha (N)\mbox{ is regular}\}
$$
and
$$
{\cal G}=\{\beta:G\hookrightarrow \Sym(N) \mbox{ such that }\beta
(G')\mbox{ is the stabilizer of } e_N\}
$$
Under this bijection, if $\alpha\in {\cal N}$ corresponds to
$\beta\in {\cal G}$, then $\alpha(N)$ is normalized by
$\lambda(G)$ if and only if $\beta(G)$ is contained in the
holomorph $\Hol(N)$ of $N$.
\end{theorem}

Let us recall that the inclusion of $\Hol(N)=N\rtimes \Aut N$ in $\Sym(N)$ is given by sending $n \in N$ to left translation by $N$ and $\sigma \in \Aut N$ to itself considered as a permutation.

\vspace{0.5cm}
In this paper we consider a finite Galois extension $K/k$ with group $G$ such that $G$ is the semi-direct product of a subgroup $G'$ and a normal subgroup $H$. For $F$ the subfield of $K$ fixed by $G'$ we prove in theorem \ref{ind} that a Hopf Galois structure on $F/k$ of type $N_1$ together with a Hopf Galois structure on $K/F$ of type $N_2$ induces a Hopf Galois structure on $K/k$ of type the direct product $N_1\times N_2$. We shall call such Hopf Galois structures induced. We shall refer to a Hopf Galois structure of type a direct product as an split Hopf Galois structure. In theorem \ref{char}, we characterize the split Hopf Galois structures of $K/k$ which are induced. In section 3 we give examples of induced and split non-induced Hopf Galois structures. In section 4 we determine the number of induced Hopf Galois extensions for some Galois extensions and obtain some cases in which all split structures are induced.

\section{Induced Hopf Galois structures}

\begin{theorem}\label{ind}
 Let $K/k$ be a finite Galois field extension, $G=\Gal(K/k)$ and let $F$ be a
 field with $k \subset F \subset K$ such that $G'=\Gal(K/F)$ has a normal complement in $G$. Let $r=[K:F]$, $t=[F:k]$, $n=[K:k]$ and assume that $N_1 \subset S_t$ gives $F/k$ a Hopf Galois structure and $N_2 \subset S_r$ gives $K/F$ a Hopf Galois structure. Then $N_1 \times N_2 \subset S_t \times S_r \subset S_n$ gives $K/k$ a  Hopf Galois structure.
\end{theorem}

\begin{proof} The action of $G'$ on itself by left translation
gives rise to a morphism $\lambda_{r}: G'\to S_r$. Let $\widetilde{F}$ be the normal closure of $F/k$ in $K$, $G''=\Gal(K/\widetilde{F})$. Then $\Gal(\widetilde{F}/k)\simeq G/G''$, $\Gal(\widetilde{F}/F)\simeq G'/G''$ and the action of $G/G''$ by left translation on the left cosets $(G/G'')/(G'/G'')$ induces an action of $G$ on the left cosets $G/G'$, since the sets $(G/G'')/(G'/G'')$ and $G/G'$ are in bijection with each other. This gives rise to a morphism $\lambda_t:G \to S_t$.

Let $H=\{x_1, \dots, x_t\}$ be a normal complement of $G'$ in $G$. Then $x_1, \dots, x_t$ is a left transversal for  $G/G'$.  The action of $g\in G$ on $G/G'$ is given by $g\cdot x_iG'= x_{\lambda_t(g)(i)} G'$.
Let $G'=\{ y_1, \dots, y_r \}$. The action of $g' \in G'$ is given by $g' \cdot y_j= y_{\lambda_r(g')(j)}$.

Now, since $N_1 \subset S_t$ gives $F/k$ a Hopf Galois structure,
$N_1$ is normalized by $\lambda_t(G)$ and, since $N_2 \subset S_r$ gives $K/F$ a Hopf Galois structure, $N_2$ is normalized by $\lambda_r(G')$. That is, for every $a\in N_1,\  g \in G$, there exists $a'\in N_1$ such that $\lambda_t(g)a=a'\lambda_t(g)$, and for every
$b\in N_2,\  g' \in G'$, there exists $b'\in N_2$ such that $\lambda_r(g')b=b'\lambda_r(g')$.

Now we have $G=\{x_iy_j, 1\leq i \leq t, 1 \leq j \leq r\}$ and, for $g \in G, g=xy, x \in H, y \in G'$, we have $yx_iy^{-1} \in H$, since $H\vartriangleleft G$, hence $yx_i=x_{\lambda_t(y)(i)} y$. The action of  $G$ on itself by left translation is then given by

\begin{equation}\label{act} gx_iy_j=x(yx_i)y_j=xx_{\lambda_t(y)(i)} y y_j = x_{\lambda_t(g)(i)} y_{\lambda_r(y)(j)},
\end{equation}

\noindent and induces a monomorphism

$$\lambda:G \hookrightarrow S_{n}=\Sym(\{1,\dots,t\}\times\{1,\dots,r\}).$$

We recall that we have an injective morphism $\iota: S_t\times S_r \hookrightarrow S_{tr}$. The element
 $(\sigma,\tau)$ in $S_t\times S_r$ is sent to the permutation of $S_{tr}$ given by
$$\begin{array}{lccc} (\sigma,\tau): &\{ 1,\dots,t \}\times \{ 1,\dots,r\} & \rightarrow & \{ 1,\dots,t \}\times \{ 1,\dots,r\} \\ & (i_1,i_2) & \mapsto & (\sigma(i_1),\tau(i_2))\end{array}$$

\noindent If $N_1$ is a regular subgroup of $S_t$ and $N_2$ is a regular subgroup of $S_r$, then under the above monomorphism, $N_1\times N_2$ is a regular subgroup of   $S_{n}$: it is transitive and its order is $tr=n$. Let us check that, if $N_1$ is normalized by $\lambda_t(G)$ and $N_2$ is normalized by $\lambda_s(G')$, then $N_1\times N_2$ is normalized by $\lambda(G)$. For $g=xy \in G, x \in H, y \in G', a \in N_1, b \in N_2$, we have
\begin{equation}\label{G}
\begin{array}{lll}
\lambda(g)(a,b)&=& (\lambda_t(g)a,\lambda_r(y)b)=(a'\lambda_t(g),b'\lambda_r(y)) \\
&=& (a',b')(\lambda_t(g),\lambda_r(y))=(a',b')\lambda(g).
\end{array}
\end{equation}
\end{proof}

\begin{remark} In the case $K=\widetilde{F}$, the result in Theorem \ref{ind} follows from Theorem 6.1 in~\cite{CRV2}.
\end{remark}

\begin{definition} A Hopf Galois structure on a Galois extension $K/k$ with Galois group $G$ will be called \emph{induced} if it is obtained as in theorem \ref{ind} for some field $F$ with $k\subsetneq F \subsetneq K$ and given Hopf Galois structures on $F/k$ and $K/F$. It will be called \emph{split} if the corresponding regular subgroup of $Sym(G)$ is the direct product of two nontrivial subgroups.
\end{definition}

\begin{corollary}\label{cor}
A Galois extension $K/k$ with Galois group $G=H\rtimes G'$ has at least one split Hopf Galois structure of type
$H\times G'$.
\end{corollary}

\begin{proof} Let $F=K^{G'}$ and let $\widetilde{F}$ be the normal closure of $F$ in $K$. Then $K/F$ is Galois with group $G'$ and $F/k$ is almost classically Galois of type $H$ since $H$ is a normal complement of $\Gal(\widetilde{F}/F)$ in $\Gal(\widetilde{F}/k)$. By theorem \ref{ind}, these two Hopf Galois structures induce a Hopf Galois structure on $K/k$ of type $H\times G'$.
\end{proof}

\begin{remark} Let us note that under the action of $G$ given by (\ref{G}), both $N_1$ and $N_2$ are $G$-stable subgroups of $N_1 \times N_2$.

\end{remark}

Taking into account Theorem \ref{theoB}, we can reformulate the construction of induced Hopf Galois structures in terms of holomorphs.
The regular subgroup $N_1$ of $S_t$ gives $F/k$ a Hopf Galois structure if and only if there is a monomorphism $\varphi_1: G \rightarrow \Hol(N_1)$ such that $\varphi_1(G')$ is the stabilizer of $1_{N_1}$ and the regular subgroup $N_2$ of $S_r$ gives $K/F$ a Hopf Galois structure if and only if there is a monomorphism $\varphi_2: G' \rightarrow \Hol(N_2)$ such that $\varphi_2(1_{G'})$ is the stabilizer of $1_{N_2}$. If we write an element $g \in G$ as $g=xy$, with $x \in H, y \in G'$, as above, the induced Hopf Galois structure on $K/k$ is then given by

$$\begin{array}{cccccc} \varphi: & G & \rightarrow & \Hol(N_1)\times \Hol(N_2) & \stackrel{\iota}{\hookrightarrow} & \Hol(N_1\times N_2) \\ & g=xy &\mapsto & (\varphi_1(g),\varphi_2(y))& &  \end{array}$$

\noindent where the monomorphism $\iota$ of $\Hol(N_1)\times \Hol(N_2)$ into $\Hol(N_1\times N_2)$ is the restriction of $\iota: \Sym(N_1)\times \Sym(N_2) \rightarrow \Sym(N_1\times N_2)$ to $\Hol(N_1)\times \Hol(N_2)$ whose image is clearly contained in $\Hol(N_1\times N_2)$. Let us check that $\varphi(1_G)$ is the stabilizer of $1_{N_1\times N_2}$. Indeed, for $g=xy \in G, x \in H, y \in G'$, $\varphi(g)(1_{N_1\times N_2})= 1_{N_1 \times N_2}$ is equivalent to $\varphi_1(g)(1_{N_1})= 1_{N_1}$ and $\varphi_2(y)(1_{N_2})= 1_{N_2}$ then to $g \in G'$ and $y=1_{G'}$, hence to $g=1_G$.

\vspace{0.5cm}
Given a Galois extension $K/k$ of degree $n$ with Galois group $G$ and a regular subgroup $N=N_1 \times N_2$ of $S_n$ giving $K/k$ a split  Hopf Galois structure, we want to determine under which conditions this Hopf Galois structure is induced.  From theorem \ref{ind} we know that the following conditions are necessary.

\begin{enumerate}[1)]
\item  $N_1$ and $N_2$ are $G$-stable,
\item If $F=K^{N_2}$ and $G'=\Gal(K/F)$, then $G'$ has a normal complement in $G$.
\end{enumerate}

\noindent We shall see in theorem \ref{char} that these conditions are also sufficient.
A first step in this direction is the following.

\begin{proposition}\label{fs} Let $K/k$ be a finite Galois field extension, $n=[K:k]$, $G=\Gal(K/k)$. Let $K/k$ be given a Hopf Galois structure such that the corresponding regular subgroup $N$ of $S_n$ has a $G$-stable subgroup $N_2$. Let $F=K^{N_2}$, $G'=\Gal(K/F)$ and $r=[K:F]$. Then $N_2$ is a regular subgroup of $S_r$ normalized by $\lambda_r(G')$.
\end{proposition}

\begin{proof} Let $H=K[N]^G$ be the Hopf algebra giving the Hopf Galois structure on $K/k$, $H_2=K[N_2]^G$ the Hopf subalgebra of $H$ corresponding to the $G$-stable subgroup $N_2$ of $N$. By \cite{C-S} theorem 7.6, $K/F$ is a Hopf Galois extension with Hopf algebra
$H_2\otimes_k F$. By classical Galois theory, $F$ is the subfield of $K$ fixed by a subgroup $G'$ of $G$ and $G'=\Gal(K/F)$. Now $(H_2\otimes_k F)\otimes_F K$ is isomorphic to $H_2\otimes_k K$, hence to $K[N_2]$, since $H_2$ is a $K$-form of $k[N_2]$. Since $K/F$ is Galois with Galois group $G'$, by \cite{G-P} theorem 3.1, we have $H_2\otimes_k F \simeq  K[N_2]^{G'}$, hence $N_2$ is a regular subgroup of the symmetric group $S_r$ normalized by $\lambda_r(G')$.
\end{proof}

\begin{theorem}\label{char} Let $K/k$ be a finite Galois field extension, $n=[K:k]$, $G=\Gal(K/k)$. Let $K/k$ be given a split Hopf Galois structure by a regular subgroup $N$ of $S_n$ such that $N=N_1\times N_2$ with $N_1$ and $N_2$ $G$-stable subgroups of  $N$. Let $F=K^{N_2}$ be the subfield of $K$ fixed by $N_2$ and let us assume that $G'=\Gal(K/F)$ has a normal complement in $G$.  Then $K/F$ is Hopf Galois with group $N_2$ and  $F/k$ is Hopf Galois with group $N_1$. Moreover the Hopf Galois structure of $K/k$ given by $N$ is induced by the Hopf Galois structures given by $N_1$ and~$N_2$.
\end{theorem}

\begin{proof} Since $K/k$ is Hopf Galois with group $N$, we have a monomorphism

$$\begin{array}{llll}\varphi:& G & \rightarrow & \Hol(N)= N \rtimes \Aut N \\ & g & \mapsto & \varphi(g)=(n(g),\sigma(g)) \end{array}$$

\noindent such that $\varphi(1_G)$ is the stabilizer of $1_N$. If $n \in N, \sigma \in \Aut N$, we have $\sigma n \sigma^{-1}= \sigma (n)$ in $\Sym(N)$. Now, if $N_1$ and $N_2$ are $G$-stable, we have $\varphi(g)(N_i) \subset N_i$, $i=1, 2$,  for all $g \in G$, i.e. for $n_i \in N_i$, $i=1,2$, $\varphi(g) n_i \varphi(g)^{-1}= n(g) \sigma(g)n_i \sigma(g)^{-1} n(g)^{-1}= n(g) \sigma(g)(n_i) n(g)^{-1}\in N_i$ which implies $\sigma(g)(n_i) \in N_i$, since $N_i  \vartriangleleft N$. Then, for $(n,\sigma) \in \varphi(G)$ with $n =(n_1,n_2)$, we have $(n,\sigma)= \iota ((n_1, \sigma_{|N_1}),(n_2,\sigma_{|N_2}))$. Since $\varphi(G) \subset \iota (\Hol(N_1)\times \Hol(N_2))$, we obtain morphisms

$$\varphi_1:G \rightarrow \Hol(N_1), \quad \varphi_2:G' \rightarrow \Hol(N_2).$$

\noindent Since $F$ is the subfield of $K$ fixed by $N_2$ and $G'=\Gal(K/F)$, we have for an element $g \in G$, $g \in G' \Leftrightarrow \varphi(g)(1_N) \in N_2$, taking into account the definition of the bijection $b$ between $G$ and $N$ used in the proof of theorem 7.3 in \cite{Ch2}. Hence $\varphi_1(G')$ is the stabilizer of $1_{N_1}$. Now for $y \in G', \varphi_2(y)(1_{N_2})=1_{N_2} \Rightarrow \varphi_2(y)(1_{N}) \in N_1$. But we had $\varphi(y)(1_N) \in N_2$, hence $\varphi(y)(1_N)=1_N$, which implies $y=1_{G}$, so $\varphi_2(1_{G'})$ is the stabilizer of $1_{N_2}$.
\end{proof}

\section{Examples}

\subsection{Examples of induced Hopf Galois structures}
Applying corollary \ref{cor}, we obtain the following results.

\begin{itemize}
\item Galois extensions with Galois group $S_3=C_3 \rtimes C_2$ have induced Hopf Galois structures of cyclic type $C_6=C_3\times C_2$
\item Galois extensions with Galois group $D_{2n}=C_n \rtimes C_2$ have induced Hopf Galois structures of type $C_n\times C_2$
\item Galois extensions with Galois group $S_{n}=A_n \rtimes C_2$ have induced Hopf Galois structures of type $A_n\times C_2$
\item Galois extensions with Galois group $A_{4}=V_4 \rtimes C_3$ have induced Hopf Galois structures of type $V_4\times C_3$
\item Galois extensions with Galois group a Frobenius group $G=H\rtimes G'$, where $H$ is the Frobenius kernel and $G'$ a Frobenius complement, have induced Hopf Galois structures of type $H\times G'$. Let us note that Sonn \cite{S} has proved that all Frobenius groups occur as Galois groups over $\Q$.
\item Galois extensions with Galois group $\Hol(M)=M \rtimes \Aut (M)$ have induced Hopf Galois structures of type $M\times \Aut(M)$.
\end{itemize}

 A partial answer to the question of which groups are a semidirect product is given by theorem \ref{SZ} below. A \emph{Hall divisor} of an integer $n$ is a divisor $m$ of $n$ such that $(m, n/ m) = 1$.
A \emph{normal Hall subgroup} of a group $G$ is a normal subgroup $N$ such that its order $|N|$ is coprime with
its index in $G$, i.e. such that $|N|$ is a Hall divisor of $|G|$.

\begin{theorem}[Schur-Zassenhaus]\label{SZ} Let $G$ be a finite group of order $n$  and let $m$ be a Hall divisor of $n$.
If there exists a normal subgroup $N$ of $G$ with order $m$, then $N$ is a characteristic subgroup of $G$ and $N$ has a complement $H$ in $G$, i.e.
$G = N \rtimes H$. If $H$ and $H'$ are two complements of $N$ in $G$, then $H$ and $H'$ are conjugate.
\end{theorem}

We have then that $G = N \rtimes H$, with $(|N|,|H|)=1$ is equivalent to $N$ being a normal Hall subgroup of $G$. Let us consider the class of groups $G$ having (at least) one normal Hall subgroup. For $p$ a prime integer, $p$-groups do not belong to this class; groups of order $2p^k$, with
$p\ge 3$ prime, and also groups of order $4p^k$, with $p\ge 5$ prime, belong to this class since its unique $p$-Sylow subgroup is a Hall normal subgroup. Therefore a Galois extension with Galois group in one of these two last sets has induced Hopf Galois structures.

A group $G$ of order $2m$ where $m$ is odd, has a subgroup of order $m$. Indeed, by Cauchy theorem, $G$ has an element $x$ of order 2. The image of $x$ by the regular representation $\varphi: G \rightarrow S_{2m}$ is an odd permutation. Then $N:=\varphi^{-1}(\varphi(G) \cap A_{2m})$ is a subgroup of order $m$ of $G$, therefore a normal Hall subgroup. The group $G$ is either a direct or semi-direct product of $N$ and $\langle x \rangle$. Hence Galois  extensions with Galois group $G$ have at least one Hopf Galois structure of type $N\times C_2$, either Galois or induced.

\subsection{Examples of split non-induced Hopf Galois structures} Not all split Hopf Galois extensions are induced. The quaternion group $H_8$ cannot be decomposed into a semi-direct product of two groups. However a Galois extension with Galois group $H_8$ has a Hopf Galois structure of type $C_2 \times C_2 \times C_2$. Let us write

$$H_8=\langle i,j| i^4=1, i^2=j^2, ij=ji^3\rangle=\{1,i,i^2,i^3, j, ij, i^2j,i^3j\}.$$

The morphism $\lambda: H_8 \rightarrow \Sym(H_8)$ given by the action of $H_8$ on itself by left translation is determined by

$$\begin{array}{l}
\lambda(i)=(1,i,i^2,i^3)(j,ij,i^2j,i^3j)\\
\lambda(j)=(1,j,i^2,i^2j)(i,i^3j,i^3,ij).
\end{array}
$$

\noindent Then, $\lambda(H_8)$ normalizes

$$
N = \langle (1,i^2)(i,i^3)(j,i^2j)(ij,i^3j), (1,i^3)(i,i^2)(i,ij)(i^2j,i^3j),(1,i^3j)(i,j)(i^2,ij)(i^3,i^2j)\rangle
$$

\noindent which is a regular subgroup of $\Sym(H_8)$ isomorphic to $C_2\times C_2\times C_2$.

Simple groups are another example of groups which are not a semidirect product of two subgroups, hence a Galois extension with Galois group a simple group has no induced Hopf Galois structures. One may wonder if it has split Hopf Galois structures.

\section{Counting Hopf Galois structures}

In this section we determine the number of split and induced Hopf Galois structures for some Galois extensions. We obtain some cases in which all split Hopf Galois structures are induced.

\subsection{The alternating group $A_4$}
Let $K/k$ be a Galois extension with Galois group $A_4$. We have checked that such an extension has only two types of Hopf Galois structures: $A_4$ and $C_2\times
C_2\times C_3$.

In \cite{Ca} it is shown that the number $e(A_4, A_4)$ of Hopf Galois structures of type $A_4$ is equal to $10$. Let us determine the number of induced Hopf Galois structures of type
$C_2\times C_2\times C_3$.

We have a unique choice for the nontrivial normal subgroup $H$, which is the Klein subgroup $V_4=\{ id, (1,2)(3,4),(1,3)(2, 4),(1,4)(2,3)\}.$
It has four different complements in $G$
$$
G'_1=\langle (2,3,4)\rangle,\  G'_2=\langle (1,3,4)\rangle,$$
$$
G'_3=\langle (1,2,4)\rangle,\  G'_4=\langle (1,2,3)\rangle.
$$

\noindent We have $G'_i=Stab(i,A_4)$ and they are conjugate.

For a fixed complement $G'$, if  $F=K^{G'}$, then $F/k$ is a quartic extension with Galois closure $K$. Therefore, there is a unique
Hopf Galois structure for $F/k$, which is given by $\lambda(V_4)\subset \Sym(V_4)$. The extension $K/F$ is Galois  of prime degree $3$
with Galois group $G'$. This is also the unique Hopf Galois structure for $K/F$. We obtain then a unique induced Hopf Galois structure for each $G'$.
Therefore $K/k$ has four different induced Hopf Galois structures of type $C_2\times C_2\times C_3$. We obtain then

$$e(A_4, V_4\times C_3)\geq 4.$$

\subsection{Groups of order $4p$}

Let us assume that $p$ is an odd prime, $G$ a nonabelian group of order $4p$ and $K/k$  a Galois extension with Galois group $G$.
Such a group $G$ has a unique $p$-Sylow subgroup $H$ and $p$ $2$-Sylow subgroups which are isomorphic either to the cyclic group $C_4$ or to
the elementary abelian group $C_2\times C_2$.

Let $G'$ be a $2$-Sylow subgroup of $G$ and $F$ its fixed field $K^{G'}$.
Since $F/k$ has degree $p$ and $G$ is solvable,
it is known  (cf. \cite{Ch1}) that $F/k$ is Hopf Galois.
Furthermore, in this case $F/k$ is almost classically Galois and has a unique Hopf Galois structure given by the normal complement $H$ of $G'$.
The number of induced Hopf Galois structures of type
$N_1\times N_2$ of $K/k$, with $N_1 \simeq H$,  depends on the number of Hopf Galois structures of type $N_2$ of $K/F$.

The number of Hopf Galois structures for Galois extensions with group isomorphic to $H$ is known:
\begin{center}
\begin{tabular}{c|c|c}
&$N_2\simeq C_4$ & $N_2\simeq  C_2\times C_2$ \\
\hline
$H\simeq C_4$
&1&1\\
$H\simeq C_2\times C_2$ &3&1\\
\end{tabular}
\end{center}

Putting all together we obtain the following numbers of induced Hopf Galois structures for $K/k$:

\begin{center}
\begin{tabular}{c|c|c}
& Structures $C_4\times C_p$ & Structures $C_2\times C_2\times C_p$ \\
\hline\hline
2-Sylow subgroup $\simeq C_4$
& $p$ &$p$\\
2-Sylow subgroup $\simeq C_2\times C_2$ & $3p$ &$p$\\
\end{tabular}
\end{center}

According to the results in \cite{Ko} these are the numbers of split Hopf Galois structures for $K/k$ of type
 $C_4\times C_p$ or $C_2 \times C_2 \times C_p$.

Since the dihedral group of order $4p$ is $D_{4p}\simeq C_2\times D_{2p}$, the dihedral
Hopf Galois structures are also split structures and we wonder if they are also induced.

The groups $D_{4p}$ and the generalized quaternion group $Q_p$ have a center of order 2. Let $G'=Z(G)$ and $F=K^{G'}$. The extension $F/k$ is Galois of order $2p$.
If it is non abelian it admits $2+p$ Hopf Galois structures, 2 of them of type $D_{2p}$ and $p$ of them of type $C_p$ (cf.\cite{B2}).
Therefore if $G=D_{4p}$, taking $G'=Z(G)$ we obtain 2 structures of type $C_2\times D_{2p}$ and $p$ structures of type $C_2\times C_{2p}=C_2\times C_2\times C_p$.

\subsection{Groups of order $pq$}

Let us assume that $G$ is a group of order $pq$ ($p>q$) and $K/k$ is a Galois extension with group $G$.
If $q\nmid p-1$, then $pq$ is a Burnside number and $K/k$ has a unique Hopf Galois structure , the classical Galois one (cf. \cite{B}).

Assume that $q\mid p-1$. Then, the group $G$ is either cyclic or metacyclic $C_p\rtimes C_q$. In the cyclic case there are $2q-1$ different
Hopf Galois structures for $K/k$, the classical one with $N\simeq C_{pq}$ (split) and $2q-2$ structures with
$N\simeq C_p\rtimes C_q$ (nonsplit).

Let us consider the nonabelian case $G\simeq C_p\rtimes C_q$. Such a group $G$ has a unique $p$-Sylow subgroup and $p$ $q$-Sylow subgroups.
Let $G'$ be a $q$-Sylow subgroup of $G$ and $F=K^{G'}$  the corresponding intermediate field. Since $F/k$ has prime degree $p$ and $G$ is solvable,
$K/k$ is Hopf Galois (cf. \cite{Ch1}). Furthermore, in this case $K/k$ is almost classically Galois and has a unique Hopf Galois structure. On the other hand, the same unicity property is true for the Galois extension  $K/F$ which has prime degree $q$.

Therefore, for each $G'$, we obtain exactly one induced Hopf Galois structure for $K/k$ and all together we obtain
in this way $p$ induced Hopf Galois structures for $K/k$. According to Theorem 6.2 in \cite{B2},
this covers all split structures for $K/k$.

In particular, if $p$ is an odd prime and $K/k$ is a dihedral extension of degree $2p$, its Hopf Galois
structures are the two given by $G$ and $G^{opp}$ (dihedral type) and the $p$ split structures of type $C_2\times C_p$ (cyclic type),
induced by the structures of $K/F$ and $F/k$, for $F=K^{G'}$ with $G'$ ranging over the set of complements in $G$ of the cyclic subgroup of order $p$ (cf. \cite{B2} Corollary 6.5).

\subsection{Safe primes and Frobenius groups}

Let $p$ be a safe prime, that is,  a prime such that $p=2q+1$ with $q$ also a prime.
Let   $G$ be a Frobenius group of order $p(p-1)$ and $K/k$ a Galois extension with group $G$.
This group $G$ has a unique $p$-Sylow subgroup and $p$ conjugate subgroups of order $p-1=2q$.
If $G'$ is one of these subgroups and $F=K^{G'}$, then $K/k$ is the Galois closure of $F/k$.

Again, $F/k$ has a unique Hopf Galois structure, with $N\simeq C_p$.
The Galois extension $K/F$ has Galois group isomorphic to $C_{2q}$.  If $q=2$ it has 2 different Hopf Galois structures,   one of them of
cyclic type $C_4$ and the other one of dihedral type $V_4$. We obtain then 5 induced structures of type $C_4\times C_5$ and 5 induced  structures of type $V_4\times C_5$ for a Frobenius extension of degree 20. This case was already treated when we considered groups of order $4p$ and we know that there are no further split structures.

If $q$ is odd, we know from \cite{B2} (Corollary 6.4) that $K/F$ has 3 Hopf Galois structures, namely the classical one and 2 of dihedral type.
We obtain then $p$ induced structures of type $C_{p-1}\times C_p$ and $2p$ induced structures of type $D_{p-1}\times C_p$ for a Frobenius group $F_{p(p-1)}$ with $p$ a safe prime $>5$. We may wonder if in this case these are all split structures.

\end{document}